\documentclass[12pt,a4paper]{article}
\usepackage{fancyhdr}
\usepackage[british]{babel}

\usepackage{amssymb, amsmath, amsthm, amsfonts}
\usepackage[usenames,dvipsnames,svgnames,table]{xcolor}
\usepackage{graphicx,tikz,caption,subcaption}
\usetikzlibrary{patterns, shapes}
\usepackage{hyperref}
\usepackage{enumitem,verbatim}
\usepackage{multicol}
\usepackage{float}
\usepackage{cleveref}
\usepackage[margin=2.5cm]{geometry}
\usepackage{dsfont}

\definecolor{gray}{rgb}{0.25, 0.25, 0.25}

\newtheorem{theorem}{Theorem}[section]

\newtheorem{conjecture}[theorem]{Conjecture}

\newtheorem*{theorem*}{Theorem}

\theoremstyle{definition}

\theoremstyle{plain}

\newtheorem{prop}[theorem]{Proposition}

\theoremstyle{definition}

\newtheorem{definition}[theorem]{Definition}

\hypersetup{
    colorlinks=true,
    linkcolor=black,
    citecolor=black,
    urlcolor=blue
}
\title{Sidorenko-Type Inequalities for Even Subdivisions over Finite Abelian Groups}
\author{Yuqi Zhao\thanks{Email: \texttt{yuqi.zhao012@gmail.com}}}
\date{}
\begin{document}

\maketitle

\begin{abstract}
Sidorenko’s conjecture states that, for every bipartite graph $H$ and any host graph $G$, the homomorphism density from $H$ to $G$ is asymptotically at least as large as in a quasirandom host graph with the same edge density as $G$. While the conjecture remains still very open, Szegedy showed that it suffices to verify the inequality when the host graph is a Cayley graph over a finite group.

In this paper, we prove that Sidorenko's conjecture holds for all even subdivisions of arbitrary graphs when the host graph is a Cayley graph over an abelian group. That is, if each edge of a graph is replaced by a path of even length (allowing different lengths for different edges), then the resulting graph satisfies the Sidorenko's inequality in any abelian Cayley host graph.

Our approach reduces the homomorphism count to the evaluation of certain averages over solution sets of linear systems over finite abelian groups, and proceeds using Fourier-analytic techniques. 
\end{abstract}
\section{Introduction}
One of the central problems in extremal graph theory is to determine the minimum possible number of copies of a given graph $H$ in graphs with a fixed edge density. A prominent conjecture, proposed independently by Sidorenko \cite{sidorenko1993correlation, sidorenko1991inequalities} and Erd\H{o}s and Simonovits \cite{erdos1984cube}, asserts that for every bipartite graph $H$, the minimum number of such copies is asymptotically attained by the random graph with the same edge density. 

For any graph $G$, let $V(G)$ denote its \emph{vertex set}, $E(G)$ denote its \emph{edge set}, and let $v(G) = |V(G)|$ and $e(G) = |E(G)|$ represent the number of vertices and edges, respectively. Formally, for graphs $H$ and $G$, a \emph{homomorphism} from $H$ to $G$ is a function $f: V(H) \to V(G)$ such that $f(u)f(v) \in E(G)$ whenever $uv \in E(H)$. Let $\hom(H, G)$ denote the number of homomorphisms from $H$ to $G$. 

The \emph{density} of $H$ in $G$ is defined as
\[
t(H, G) = \frac{\hom(H, G)}{v(G)^{v(H)}}.
\]
Sidorenko's conjecture is stated as follows:

\begin{conjecture}[Sidorenko's Conjecture]\label{SidorenkoConj}
For every bipartite graph $H$ and every graph $G$, we have 
    \begin{equation}\label{SidorenkoConje}
        t(H,G)\geq t(K_2,G)^{e(H)}.
    \end{equation}    
\end{conjecture}

A graph $H$ satisfying this inequality is said to have \emph{Sidorenko's property}. When $H$ is a connected graph with $v(H) > e(H)$, it must be a tree, and Sidorenko~\cite{sidorenko1993correlation} proved that all trees satisfy Conjecture~\ref{SidorenkoConj}. If $H$ is disconnected with connected components $H_1, \dots, H_k$, then $t(H, G) = \prod_{i=1}^{k} t(H_i, G)$, and thus the conjecture reduces to proving the case where $H$ is connected with $e(H) \geq v(H)$.

In addition to trees, Sidorenko~\cite{sidorenko1993correlation} also proved that even cycles and complete bipartite graphs satisfy the conjecture. Subsequent works have extended the list of bipartite graphs known to have Sidorenko's property. Hatami~\cite{hatami2010graph} established the conjecture for hypercubes, and Conlon, Fox, and Sudakov~\cite{conlon2010approximate} verified it for bipartite graphs in which one vertex is complete to the other side. Conlon, Kim, Lee, and Lee~\cite{conlon2018some} proved the conjecture for a broad class of graphs known as strongly tree-decomposable graphs. More recently, Im, Li, and Liu~\cite{im2024sidorenko} showed that Sidorenko's property holds for graphs obtained by replacing each edge of an arbitrary graph with a generalized theta graph consisting of even-length paths.

Given a finite group $G$ and a symmetric subset $S \subseteq G \setminus \{e\}$ (i.e., $S = S^{-1}$), the \emph{Cayley graph} $\mathrm{Cay}(G, S)$ is the graph with vertex set $G$, where two elements $x, y \in G$ are adjacent if and only if $x^{-1}y \in S$. Equivalently, the edge set consists of all pairs $\{x, xs\}$ with $x \in G$ and $s \in S$.

In \cite{szegedy2015sparse}, Szegedy proved that to establish Sidorenko's conjecture, it suffices to verify inequality~\eqref{SidorenkoConje} when the host graph $G$ is vertex-transitive and edge-transitive. Note that \cite{lovasz2015automorphism} for any vertex transitive graph $G$, one can obtain a Cayley graph $G'$ from $G$ by replacing every vertex by $m$ vertices and every edge by a complete bipartite graph $K_{m,m}$, where the value $m$ is the size of the stabilizer of a vertex in $G$ in the automorphism group of graph $G$. Thus, we have the following proposition:

\begin{prop}
To prove Sidorenko's Conjecture, it is enough to show that for every bipartite graph $H$, every finite group $G$, and every symmetric subset $S \subseteq G \setminus \{e\}$, the inequality holds:
\[
   t(H, \mathrm{Cay}(G, S)) \geq t(K_2, \mathrm{Cay}(G, S))^{e(H)}.
\]
\end{prop}

In this paper, we focus on the case where the host graph is a Cayley graph over a finite abelian group. This setting is particularly well suited to harmonic analysis techniques and exhibits a rich algebraic structure. A closely related line of investigation was carried out by Cho, Conlon, Lee, Skokan, and Versteegen~\cite{cho2024norming}, who studied the norming property of graphs in Cayley graphs over additive groups of finite fields, i.e., groups of the form $\mathbb{F}_q^n$, where $q$ is a prime power and $n$ a parameter. 

An important class of graphs that has attracted significant attention in connection with Sidorenko's conjecture is that of \emph{even subdivisions}. Given a graph~$H$, an \emph{even subdivision} of~$H$ is obtained by replacing each edge with a path of even length (not necessarily the same length for different edges). It is easy to observe that any such subdivision is bipartite, regardless of whether the original graph~$H$ is.

The Sidorenko property for various families of even subdivisions has been studied in depth in several works, including the recent papers of Conlon, Kim, Lee, and Lee~\cite{conlon2018some}, and of Im, Li, and Liu~\cite{im2024sidorenko}. For any graph~$H$, we refer to the graph obtained by replacing each edge with a path of length~$2$ as the \emph{standard subdivision} of~$H$. Very recently, Cho, Conlon, Lee, Skokan, and Versteegen \cite{cho2024norming} proved that in $\mathbb{F}_q^n$, subdivision of a weakly norming system is also weakly norming, and that the equations corresponding to a subdivision is positive. Besides, the following result is known:

\begin{theorem}[\cite{conlon2018some}]
\leavevmode
\begin{enumerate}
    \item For every positive integer~$k$, the standard subdivision of the complete graph~$K_k$ has Sidorenko's property.
    \item If a graph~$H$ has Sidorenko's property, then the standard subdivision of~$H$ also has Sidorenko's property.
\end{enumerate}
\end{theorem}

While these works establish the Sidorenko property for standard subdivisions under specific host graphs, it is still open for standard subdivisions under general host graphs. In this paper, we show that the Sidorenko inequality holds for arbitrary even subdivisions when the host graph is a Cayley graph over a finite abelian group. Our main result is as follows:

\begin{theorem}[Main result]\label{main}
Let $G$ be a finite abelian group and $S$ be a symmetric subset of $G$. Let $H_0$ be an arbitrary graph, and let $H$ be an even subdivision of $H_0$. Then the following inequality holds:
\[
        t(H,\mathrm{Cay}(G, S))\geq t(K_2,\mathrm{Cay}(G, S))^{e(H)}.
\]
\end{theorem}

\section{Preliminaries}
\subsection{Fourier-analytic approach}
Let $G$ be a finite abelian group, we write $\widehat{G}$ for the dual group of $G$, i.e., the group of homomorphisms from $G$ to the multiplicative group $\mathbb{C}$ of complex numbers, which is easily seen to be isomorphic to $G$ itself.

Note that every finite abelian group $G$ is isomorphic to the group $\mathbb{Z}_{N_1} \times \cdots \times \mathbb{Z}_{N_k}$ for some positive integers $N_1, \dots, N_k$. Now let $G = \mathbb{Z}_{N_1} \times \cdots \times \mathbb{Z}_{N_k}$ for some positive integers $N_1, \dots, N_k$. For every $a = (a_1, \dots, a_k) \in G$, define $\chi_a \in \widehat{G}$ as the product of the characters $\chi_{a_1}, \dots, \chi_{a_k}$ of the groups $\mathbb{Z}_{N_1}, \dots, \mathbb{Z}_{N_k}$ applied to the coordinates of $x \in G$, respectively. More precisely,
\[
\chi_a(x) = \prod_{j=1}^k e^{2\pi i \frac{a_j x_j}{N_j}}.
\]
Let $f: G\to \mathbb{C}$ be a function, defining its Fourier transform as:
\[
\widehat{f}(\chi) = \langle f, \chi \rangle = \mathbb{E}[f(x)\overline{\chi(x)}].
\]
Note that for any $\chi_a \in \widehat{G}$ we have $\widehat{f}(\chi_a) =\overline{\widehat{f}((\chi_a)^{-1}) }=\overline{\widehat{f}(\chi_{-a}) }$ thus  
\begin{equation}\label{pair}
    \widehat{f}(\chi_a)\widehat{f}(\chi_{-a})\geq 0. 
\end{equation}
The Fourier coefficient $\widehat{f}(0)$ is of particular importance, as
\[
\widehat{f}(0) = \mathbb{E}[f(x)].
\]
In particular, if $\mathbf{1}_A$ is the indicator function of a subset $A \subseteq G$, then
\[
\widehat{\mathbf{1}_A}(0) = \frac{|A|}{|G|},
\]
which corresponds to the density of the set $A$ in $G$.

We have the Fourier inversion formula
\[
f(x) = \sum_{a \in G} \widehat{f}(a) \chi_a(x),
\]
and that this expansion of $f$ as a linear combination of characters is unique. 

Following a similar argument as in~\cite{cho2024norming}, we have the following:

\begin{prop}\label{matrix}
Let $L$ be an $m \times k$ matrix with entries in $\{0, 1, -1\}$, where $m \leq k$ and $\operatorname{rank}(L) = m$. We have
\[
\mathbb{E}_{(x_1, \dots, x_k) \in \ker(L)} f(x_1) \cdots f(x_k) = \sum_{(\xi_1, \dots, \xi_m) \in \widehat{G}^m} \prod_{j=1}^k \widehat{f} \left( \sum_{i=1}^m L_{ij} \xi_i \right).
\]
\end{prop}

\begin{proof}
Let $G$ be a finite abelian group and $L\in\{0,1,-1\}^{\,m\times k}$.
For $\boldsymbol{x}=(x_1,\dots,x_k)\in G^{k}$ and every $i\in [m]$ write
\[
   L_i(\boldsymbol{x}) \;:=\; \sum_{j=1}^{k} L_{ij}x_j \in G
   \qquad(1\le i\le m).
\]

Since 
\[
\frac{1}{|G|}\sum_{\xi\in\widehat{G}}\chi_{\xi}(y)
=\delta_{0}(y)
=\begin{cases}
1, & y = 0,\\[2pt]
0, & y \neq 0,
\end{cases}
\]
we have
\[
   \mathbf 1_{\ker L}(\boldsymbol{x})
   \;=\;
   \prod_{i=1}^{m}\delta_{0}\!\bigl(L_i(\boldsymbol{x})\bigr)
   \;=\;
   |G|^{-m}
   \sum_{(\xi_1,\dots,\xi_m)\in\widehat{G}^{\,m}}
   \prod_{i=1}^{m}\chi_{\xi_i}\bigl(L_i(\boldsymbol{x})\bigr).
\]

Note that, for each $i\in[m]$,
\[
   \chi_{\xi_i}\!\bigl(L_i(\boldsymbol{x})\bigr)
   = \prod_{j=1}^{k}\chi_{\xi_i}\!\bigl(L_{ij}x_j\bigr)
   = \prod_{j=1}^{k}\chi_{\xi_i}(x_j)^{\,L_{ij}}.
\]
Hence
\[
   \prod_{i=1}^{m}\chi_{\xi_i}\!\bigl(L_i(\boldsymbol{x})\bigr)
   = \prod_{i=1}^{m}\,\prod_{j=1}^{k}\chi_{\xi_i}(x_j)^{\,L_{ij}}
   = \prod_{j=1}^{k}\,\prod_{i=1}^{m}\chi_{\xi_i}(x_j)^{\,L_{ij}}
   = \prod_{j=1}^{k}\chi_{\sum_{i=1}^{m}L_{ij}\xi_i}(x_j).
\]

Denote by $\mathbb{E}_{\ker L}$ the uniform average over $\ker L$.
Writing this average as a sum over $G^{k}$ and using the above formula gives
\[
   \mathbb{E}_{\boldsymbol{x}\in\ker L}\prod_{j=1}^{k}f(x_j)
   =\frac{1}{|\ker L|}\sum_{\boldsymbol{x}\in G^{k}}
      \mathbf 1_{\ker L}(\boldsymbol{x})\prod_{j=1}^{k}f(x_j)
   =\sum_{(\xi_1,\dots,\xi_m)\in\widehat{G}^{\,m}}
     \prod_{j=1}^{k}\widehat f\!\Bigl(\sum_{i=1}^{m}L_{ij}\xi_i\Bigr),
\]
thus the proposition holds.

\end{proof}

\subsection{The circuit matrix}\label{subsec:circuit-matrix}

Let $H = (V, E)$ be a finite \emph{connected} graph with $|V| = n$ and $|E| = k$.  
Fix a spanning tree $T \subseteq E$. For each edge
\(
e \in E \setminus T,
\)  
there exists a unique \emph{fundamental cycle}  
\(C_e \subseteq E\),  
formed by adding the edge $e$ to the tree $T$.

\begin{definition}[Circuit matrix]
Let $H$ be a bipartite graph and $m=k-n+1$ be the \emph{cyclomatic number} of~$H$
(the minimum number of edges whose removal makes $H$ acyclic).
Order the chords
\(E\setminus T=\{e_{1},\dots,e_{m}\}\)
and their fundamental cycles
\(\mathcal{B}=\{C_{e_{1}},\dots,C_{e_{m}}\}\). For each $i\in [m]$ giving a proper edge coloring of cycle $C_{e_{i}}$ by using $2$ colors $i_1$ and $i_{-1}$.
The \emph{circuit matrix} of $H$ is the matrix
\[
L
   \;=\;
   \bigl(L_{ij}\bigr)
   \;\in\;\{-1,0,1\}^{\,m\times k},
\quad
m=k-n+1,
\]
whose rows are indexed by $\mathcal{B}$, columns are indexed by $E$, and whose entries are
\[
L_{ij}\;=\;
\begin{cases}
+1 &\text{if $j\in C_{e_i}$ and it is colored by $i_1$},\\[3pt]
-1 &\text{if $j\in C_{e_i}$ and it is colored by $i_{-1}$},\\[3pt]
0  &\text{if $j\notin C_{e_i}$}.
\end{cases}
\]
\end{definition}

The circuit matrix formalism will be essential in the sequel when we translate combinatorial constraints on homomorphisms into linear conditions suitable for Fourier–analytic treatment.

\begin{prop}\label{trantomatrix}
    Let $G$ be a finite abelian group and $S$ be a symmetric subset of $G$, and $f=\mathbf{1}_S$. Let $H$ be a bipartite graph on $n$ vertices and $k$ edges, and $L$ be its circuit matrix. The following holds:
    \[
    t(H,\mathrm{Cay}(G, S))=\mathbb{E}_{(x_1, \dots, x_k) \in \ker(L)} f(x_1) \cdots f(x_k).
    \]
\end{prop}
\begin{proof}
Let $V(H) = W \cup U$ be a fixed bipartition with $W = \{1, \dots, n_1\}$ and $U = \{n_1 + 1, \dots, n\}$. Then, by the definition of homomorphism density,
\[
t(H, \mathrm{Cay}(G, S))
= \mathbb{E}_{(v_1, \dots, v_n) \in G^n}
\prod_{\substack{i \in W,\; j \in U \\ \{i,j\} \in E(H)}}
f(v_i - v_j).
\]
Now consider the \emph{signed incidence matrix} $M \in \mathbb{Z}^{k \times n}$ of $H$, defined by:
\[
m_{e,v} =
\begin{cases}
+1 & \text{if edge $e$ is incident to vertex $v \in W$}, \\
-1 & \text{if edge $e$ is incident to vertex $v \in U$}, \\
0 & \text{otherwise}.
\end{cases}
\]
We index the rows of $M$ by the edges $e_1, \dots, e_k$ and the columns by the vertices $1, \dots, n$.

Given a vector $\boldsymbol{x} = (x_1, \dots, x_n) \in G^n$, the vector $M\boldsymbol{x} \in G^k$ has entries
\[
(M\boldsymbol{x})_e = x_i - x_j \quad \text{for each edge } e = \{i, j\} \text{ with } i \in W,\, j \in U.
\]
Hence,
\[
t(H, \mathrm{Cay}(G, S)) = \mathbb{E}_{\boldsymbol{x} \in G^n} \prod_{e \in E(H)} f\big((M\boldsymbol{x})_e\big).
\]

Let $\Phi: G^n \to G^k$ be the map $\Phi(\boldsymbol{v}) := M\boldsymbol{v}$. Then we have:
\[
t(H, \mathrm{Cay}(G, S)) = \mathbb{E}_{\boldsymbol{x} \in \mathrm{Im}(M)} f(x_1)\cdots f(x_k).
\]

It remains to show that $\mathrm{Im}(M) = \ker(L)$. First, since each row of $L$ corresponds to a cycle in $H$, and the image of $M$ clearly satisfies all cycle constraints, we have:
\[
\mathrm{Im}(M) \subseteq \ker(L).
\]

Moreover, both spaces have the same dimension:
\[
\dim (\mathrm{Im}(M)) = n - 1 = \dim (\ker(L)),
\]
since $H$ is connected and $\operatorname{rank}(L) = k - (n - 1)$. Therefore, equality holds:
\[
\mathrm{Im}(M) = \ker(L).
\]

This completes the proof.
\end{proof}

\section{Proof of the main result}
Let $H_0$ be an arbitrary graph with edge set $\{e_1, e_2, \dots, e_k\}$, and let $\{m_1, m_2, \dots, m_k\}$ be a set of positive integers. Define $H_1$ to be the graph obtained from $H_0$ by replacing each edge $e_i$ with a path of length $m_i$ for every $i \in [k]$. Let $H$ be the graph obtained from $H_0$ by replacing each edge $e_i$ with a path of length $2m_i$. 

Note that $H$ is the standard subdivision of $H_1$. Therefore, to prove Theorem~\ref{main}, it suffices to establish the following proposition:

\begin{prop}
Let $G$ be a finite abelian group and $S$ be a symmetric subset of $G$. Let $H_1$ be an arbitrary graph, and let $H$ be its standard subdivision. Then the following inequality holds:
\[
        t(H,\mathrm{Cay}(G, S))\geq t(K_2,\mathrm{Cay}(G, S))^{e(H)}.
\]
\end{prop}
\begin{proof}
Let $L_1$ be the circuit matrix of $H_1$, with $m$ rows and $k$ columns. Construct an $m \times 2k$ matrix $L$ by replacing each column $v$ of $L_1$ with two columns $v$ and $-v$. Since $H$ is obtained from $H_1$ by subdividing each edge exactly once, each original edge is replaced by a path of length 2. If such a path lies in a properly colored cycle, then the two edges in the path are assigned different colors. By the definition of circuit matrix, $L$ is indeed a circuit matrix of~$H$.

Let $f=\mathbf{1}_S$. By Propositions~\ref{trantomatrix} and~\ref{matrix}, we have
\[
\begin{aligned}
    t(H,\mathrm{Cay}(G, S)) &= \mathbb{E}_{(x_1, \dots, x_{2k}) \in \ker(L)} f(x_1) \cdots f(x_{2k}) \\
    &= \sum_{(\xi_1, \dots, \xi_m) \in \widehat{G}^m} \prod_{j=1}^{2k} \widehat{f} \left( \sum_{i=1}^m L_{ij} \xi_i \right).
\end{aligned}
\]

Observe that for each $j \in [k]$, the columns $2j - 1$ and $2j$ of $L$ are negatives of each other. Therefore,
\[
\sum_{i=1}^m L_{i, 2j-1} \xi_i + \sum_{i=1}^m L_{i, 2j} \xi_i = 0,
\]
so the two Fourier coefficients are complex conjugates. By (\ref{pair}) it follows that for all $(\xi_1, \dots, \xi_m) \in \widehat{G}^m$,
\[
\begin{aligned}
\prod_{j=1}^{2k} \widehat{f} \left( \sum_{i=1}^m L_{ij} \xi_i \right)
&= \prod_{j=1}^{k} \left( \widehat{f} \left( \sum_{i=1}^m L_{i, 2j-1} \xi_i \right) \cdot \widehat{f} \left( \sum_{i=1}^m L_{i, 2j} \xi_i \right) \right) \\
&= \prod_{j=1}^{k} \left| \widehat{f} \left( \sum_{i=1}^m L_{i, 2j} \xi_i \right) \right|^2 \geq 0.
\end{aligned}
\]

In particular, this shows that all terms in the sum are non-negative. Thus,
\[
\begin{aligned}
\sum_{(\xi_1, \dots, \xi_m) \in \widehat{G}^m} \prod_{j=1}^{2k} \widehat{f} \left( \sum_{i=1}^m L_{ij} \xi_i \right)
&= \prod_{j=1}^{2k} \widehat{f}(0) + \sum_{(\xi_1, \dots, \xi_m) \neq (0, \dots, 0)} \prod_{j=1}^{2k} \widehat{f} \left( \sum_{i=1}^m L_{ij} \xi_i \right) \\
&\geq \widehat{f}(0)^{2k} = \left( \frac{|S|}{|G|} \right)^{2k} = t(K_2,\mathrm{Cay}(G, S))^{2k}.
\end{aligned}
\]

This completes the proof.
\end{proof}
Note that the preceding argument shows the following.  
If there exist $\varepsilon>0$ and a non-trivial character $\chi\in\widehat{G}$ such that
\[
   |\widehat{f}(\chi)| \ge \varepsilon\,\widehat{f}(0),
\]
then the Sidorenko-type inequality is strict:
\[
   t(H, \mathrm{Cay}(G, S))
   \ge
   t(K_2, \mathrm{Cay}(G, S))^{e(H)}
   (1 + \varepsilon^{e(H)}).
\]

Conversely, fix any $\varepsilon > 0$ and put $\delta = \varepsilon^{e(H)} > 0$.  
If
\[
   t(H, \mathrm{Cay}(G, S))
   \le
   t(K_2, \mathrm{Cay}(G, S))^{e(H)} (1 + \delta),
\]
then for \emph{every} character $\chi \in \widehat{G}$ we must have  
\begin{equation}\label{rationf0}
       |\widehat{f}(\chi)| \le \varepsilon\,\widehat{f}(0).
\end{equation}

It is a basic fact that for abelian Cayley graphs, the eigenvalues are given by
\[
   \lambda_\chi
   = \sum_{g \in S} \chi(g)
   = |G|\,\widehat{f}(\chi),
   \qquad \chi \in \widehat{G}.
\]
In view of \eqref{rationf0}, this implies that every non-principal eigenvalue must satisfy
\[
   |\lambda_\chi| \le \varepsilon\,\frac{|S|}{|G|}.
\]

In other words, any abelian Cayley graph that nearly minimises the homomorphism density 
$t(H, \mathrm{Cay}(G, S))$ must be spectrally quasirandom: 
all its non-trivial eigenvalues are $o(|S|/|G|)$, 
and the graph therefore behaves much like a random $d$-regular graph 
at edge density $p = |S|/|G|$.

\section{Acknowledgments}

The author is grateful to Professor David Conlon for his insightful comments and valuable suggestions.

\bibliographystyle{abbrv}
\bibliography{ref}

\end{document}